%% file: alfaCheeger_revised_pp.tex
\documentclass[a4paper,10pt,reqno]{amsart}

\usepackage{amsthm,amsmath,amsfonts,amssymb,graphics}

\def\eps{\varepsilon}
\def\C{\mathcal{C}}
\def\R{\mathbb{R}}
\def\RR{\mathcal{R}}
\def\H{\mathcal{H}}

\theoremstyle{plain}
\newtheorem{thm}{Theorem}[section]
\newtheorem{lem}[thm]{Lemma}
\newtheorem{prop}[thm]{Proposition}
\newtheorem{cor}[thm]{Corollary}

\theoremstyle{definition}
\newtheorem{defin}[thm]{Definition}

\theoremstyle{remark}
\newtheorem{rem}[thm]{Remark}

\title[On the generalized Cheeger problem and an application to 2d strips]{On the generalized Cheeger problem\\ and an application to 2d strips}

\author{Aldo Pratelli}
\email{pratelli@math.fau.de}
\author{Giorgio Saracco}
\email{giorgio.saracco@unife.it}

\begin{document}

\begin{abstract}
In this paper we consider the generalization of the Cheeger problem which comes by considering the ratio between the perimeter and a certain power of the volume. This generalization has been already sometimes treated, but some of the main properties were still not studied, and our main aim is to fill this gap. We will show that most of the first important properties of the classical Cheeger problem are still valid, but others fail; more precisely, long and thin rectangles will give a counterexample to the property of Cheeger sets of being the union of all the balls of a certain radius, as well as to the uniqueness. The shape of Cheeger set for rectangles and strips is then studied as well as their Cheeger constant.
\end{abstract}

 \hspace{-3cm}
 {
 \begin{minipage}[t]{0.5\linewidth}
 \begin{scriptsize}
 \vspace{-2cm}
 This is a pre-print of an article published in \emph{Rev. Mat. Iberoamericana}. The final authenticated version is available online at: http://dx.doi.org/10.4171/RMI/934
 \end{scriptsize}
\end{minipage} 
}

\maketitle

\section*{Introduction} 

The celebrated Cheeger problem consists in searching for sets $E\subseteq \R^n$ minimizing the ratio
\begin{equation}\label{chst}
\inf_{E \subseteq \Omega} \frac{P(E)}{|E|}\,,
\end{equation}
where $\Omega\subseteq \R^n$ is a given open set of finite volume. Throughout the paper, we will write $|E|$ to denote the volume (i.e., the Lebesgue $\mathcal{L}^n$ measure) of any Borel set $E$, while $P(E)$ will be its perimeter, that is, the $\H^{n-1}$ Hausdorff measure of its reduced boundary $\partial^* E$ (for definitions and properties of sets of finite perimeter, the reader can look in~\cite{AmFuPa}). The most interesting questions in the Cheeger problem regard existence, uniqueness and main features of the \emph{Cheeger sets}, that are the sets $E$ realizing the above infimum. Notice that any set $E$ which is a Cheeger set for some $\Omega\supseteq E$ is also a Cheeger set for itself. Since the literature on this problem is huge and well known, we do not try to give here a complete list: the interested author can find a description of the history and a good bibliography in the papers~\cite{Parini,Leonardi}.\par

In this paper, we are interested in the following generalization of the Cheeger problem,
\begin{equation}\label{CheegerPbl}
h_\alpha (\Omega) := \inf_{E \subseteq \Omega} \frac{P(E)}{|E|^{1/\alpha}}\,,
\end{equation}
for $\alpha>1$. Notice that increasing $\alpha$ means that the isoperimetric properties of the sets become ``more important''. This problem is called the ``$\alpha$-Cheeger problem'', the above ratio is called ``$\alpha$-Cheeger ratio of $E$'', and any set realizing the infimum is called an ``$\alpha$-Cheeger set''. Actually, the problem is interesting only if $\alpha<1^*=n/(n-1)$, as we will see later. This generalized Cheeger problem has been already considered in the literature, see for instance~\cite{FMP09,FiMaPr} and the references therein for some general results. However, up to our knowledge, some of the main basic properties were not studied or not proved: the main aim of the present paper is to fill this gap. In particular, we will first concentrate, in Section~\ref{apriori}, on the very basic preliminary properties of the Cheeger problem, and we check how they generalize to the $\alpha$-Cheeger problem. Then, in Section~\ref{existence} we show that the existence of an $\alpha$-Cheeger set $E$ for any set $\Omega$ is always true. Section~\ref{rectangles} is devoted to studying more extensively the case of the rectangles, and checking how their Cheeger sets behave; indeed, as we will see, long and thin rectangles are a counterexample to the known property of the classical Cheeger problem that, if $\Omega$ is convex, then its Cheeger set is unique, and it is the union of all the balls of a certain radius contained in $\Omega$. Finally, in the last Section~\ref{strips} we generalize our observations to the case of the strips, which have been studied a lot in literature lately.

\section{Preliminary properties\label{apriori}}

In this first section, we give a list of well known properties of the classical Cheeger problem, and see how to generalize them for the $\alpha$-Cheeger problem; most of them will be straightforward generalizations.

\begin{thm}\label{classical}
Let us denote by $\C^1_\Omega$ a Cheeger set for $\Omega$ and by $h_1(\Omega)$ the infimum in~(\ref{chst}). Then, the following properties hold:
\begin{itemize}
\item[{\bf 1.}] The Cheeger problem is scale invariant while the Cheeger constant is not; more precisely, for any $t>0$, the Cheeger sets in $t\Omega$ are precisely the sets $tE$ for all the Cheeger sets $E$ in $\Omega$; consequently, $h(t\Omega)=t^{-1} h(\Omega)$;
\item[{\bf 2.}] The constrained boundary of any Cheeger set, i.e. $\partial \C^1_\Omega \cap \partial \Omega$, contains at least two points;
\item[{\bf 3.}] The free boundary of any Cheeger set, i.e. $\partial \C^1_\Omega \cap \Omega$, is analytic possibly except for a closed singular set whose Hausdorff dimension does not exceed $n-8$;
\item[{\bf 4.}] The mean curvature of the free boundary is constant at every regular point and, for $n=2$, it equals $h_1(\Omega)$;
\item[{\bf 5.}] The free boundary of a Cheeger set meets $\partial\Omega$ tangentially at the regular points of $\partial \Omega$: more precisely, if the boundary of a Cheeger set cointains a point $x\in\partial\Omega$ at which the normal vector to $\partial\Omega$ is defined, then also the normal vector to the Cheeger set is defined at $x$, and the two vectors coincide;
\item[{\bf 6.}] A Cheeger set of $\Omega \subseteq \mathbb{R}^2$ can not have corners with an angle smaller than $\pi$;
\item[{\bf 7.}] If a Cheeger set exists, then any of its connected components is also a Cheeger set;
\item[{\bf 8.}] If $\Omega \subseteq \R^2$ is convex, then its Cheeger set is unique and convex; in particular, it is the union of all the balls of radius $1/h_1(\Omega)$ which are contained in $\Omega$;
\item[{\bf 9.}] If $\Omega_1 \subseteq \Omega_2$ then $h_1(\Omega_1) \geq h_1(\Omega_2)$ but the strict inclusion does not imply the strict inequality.
\end{itemize}
\end{thm}

All the previous properties are well known, a discussion can be found for instance in~\cite{LeoPra}. The only result of this section is the following, where we show that all the above properties generalize, with minor changes, to the case of the $\alpha$-Cheeger problem; in particular, a stronger version of $7.$ holds, while only a weaker version of $8.$ is true. We also add a couple of ``new'' properties, where we compare the Cheeger constants relative to different powers $\alpha_1$ and $\alpha_2$, which of course make no sense in the classical case. For the sake of shortness, here and in the rest of the paper we will always write $\C^\alpha_\Omega$ to denote an $\alpha$-Cheeger set in $\Omega$.

\begin{thm}\label{properties}
Let $\Omega, \Omega_1, \Omega_2 \subseteq \mathbb{R}^n$ be open connected sets, let $\alpha, \alpha_1, \alpha_2$ be in $(1,1^*)$ and let $B_1$ be the volume-unitary $n$-dimensional ball. Then the following are true:
\begin{itemize}
\item[{\bf 1.}] The $\alpha$-Cheeger problem is scale invariant while the $\alpha$-Cheeger constant is not; more precisely, for any $t>0$, the $\alpha$-Cheeger sets in $t\Omega$ are precisely the sets $tE$ for all the $\alpha$-Cheeger sets $E$ in $\Omega$; consequently, $h_\alpha(t\Omega)=t^{n-1-\frac{n}{\alpha}} h_\alpha(\Omega)$;
\item[{\bf 2.}] The constrained boundary of an $\alpha$-Cheeger set, i.e. $\partial \C^\alpha_\Omega \cap \partial \Omega$, contains at least two points;
\item[{\bf 3.}] The free boundary of an $\alpha$-Cheeger set, i.e. $\partial \C^\alpha_\Omega \cap \Omega$, is analytic possibly except for a closed singular set whose Hausdorff dimension does not exceed $n-8$;
\item[{\bf 4.}] The mean curvature of the free boundary is constant at every regular point and, for $n=2$, it equals
\begin{equation}\label{alphacurvature}
\frac{h_\alpha(\Omega)}{\alpha}\,|\C^\alpha_\Omega|^{\frac{1}{\alpha}-1};
\end{equation}
\item[{\bf 5.}] The free boundary of an $\alpha$-Cheeger set meets $\partial\Omega$ tangentially at the regular points of $\partial \Omega$: more precisely, if the boundary of an $\alpha$-Cheeger set cointains a point $x\in\partial\Omega$ at which the normal vector to $\partial\Omega$ is defined, then also the normal vector to the $\alpha$-Cheeger set is defined at $x$, and the two vectors coincide;
\item[{\bf 6.}] An $\alpha$-Cheeger set of $\Omega \subseteq \mathbb{R}^2$ can not have corners with an angle smaller than $\pi$;
\item[{\bf 7.}] Any $\alpha$-Cheeger set is connected;
\item[{\bf 8.}] If $\Omega \subseteq \mathbb{R}^2$ is convex then any $\alpha$-Cheeger set is convex;
\item[{\bf 9.}] If $\Omega_1 \subseteq \Omega_2$ then $h_\alpha(\Omega_1) \geq h_\alpha(\Omega_2)$ but the strict inclusion does not imply the strict inequality;
\item[{\bf 10.}] If $B_1\subseteq\Omega$ and $\alpha_2 > \alpha_1$ then $h_{\alpha_2} (\Omega) \geq h_{\alpha_1} (\Omega)$;
\item[{\bf 11.}] If $|\Omega| \leq1$ and $\alpha_2 > \alpha_1$ then $h_{\alpha_2} (\Omega) \leq h_{\alpha_1} (\Omega)$.
\end{itemize}
\end{thm}
\begin{proof}
{\bf 1.} Given any set $\Omega\subseteq \R^n$, let us consider its rescaling $t\Omega$ for some $t>0$: there is a natural bijection between the subsets of $\Omega$ and the subsets of $t\Omega$ given by $E \mapsto tE$. Thus we have
\begin{equation}\label{asbefore}
h_\alpha(t\Omega) = \inf_{F\subseteq t\Omega} \frac{P(F)}{|F|^{1/\alpha}} =  \inf_{E\subseteq\Omega}  \frac{P(tE)}{|tE|^{1/\alpha}}  = t^{n-1-\frac{n}{\alpha}}\inf_{E\subseteq \Omega}  \frac{P(E)}{|E|^{1/\alpha}} = t^{n-1-\frac{n}{\alpha}} h_\alpha(\Omega),
\end{equation}
which directly tells us that $E$ is an $\alpha$-Cheeger set for $\Omega$ if and only if $tE$ is an $\alpha$-Cheeger set for $t\Omega$. Note that we get an estimate not depending on $t$ if $\alpha = 1^*$, the reason for that will appear clear in Section~\ref{existence}.\\
{\bf 2.} First of all, assume the existence of an $\alpha$-Cheeger set $\C^\alpha_\Omega$ such that $\C^\alpha_\Omega \subset \subset \Omega$. Then, the rescaled set $t\C^\alpha_\Omega$ would still be contained in $\Omega$ for some $t>1$, and this is against the minimality of $\C^\alpha_\Omega$ since for any $\alpha<n/(n-1)$ it is
\[
\frac{P(t\C^\alpha_\Omega)}{|t\C^\alpha_\Omega|^{1/\alpha}} = \frac{t^{n-1}P(\C^\alpha_\Omega)}{{t^{n/\alpha}}|\C^\alpha_\Omega|^{1/\alpha}} = t^{n-1-\frac n\alpha}\,\frac{P(\C^\alpha_\Omega)}{|\C^\alpha_\Omega|^{1/\alpha}} < \frac{P(\C^\alpha_\Omega)}{|\C^\alpha_\Omega|^{1/\alpha}} = h_\alpha (\Omega)\,.
\]
Suppose now that the boundary of some $\alpha$-Cheeger set $\C^\alpha_\Omega$ touches $\partial \Omega$ in a single point, and assume for simplicity this point to be the origin of $\R^n$. Fix then two small constants $\eps\ll \delta \ll 1$, and consider the modified set
\[
\widetilde\C = \C_1 \cup \big[1,1+\eps\big] \Gamma\cup (1+\eps)\C_2\,,
\]
where
\begin{align*}
\C_1=\C^\alpha_\Omega \cap B_\delta,\, && \C_2=\C^\alpha_\Omega\setminus B_\delta\,, && \Gamma = \C^\alpha_\Omega \cap \partial B_\delta\,,
\end{align*}
and $B_\delta=\{x\in \R^n:\, |x|<\delta\}$ is the ball with radius $\delta$ centered at the origin. Notice that $\widetilde \C\subseteq\Omega$ as soon as $\eps$ and $\delta$ are small enough. Thanks to the well-known Vol'pert Theorem (see for instance~\cite[Th~3.108]{AmFuPa}), for almost every $\delta>0$ one has that $\partial^* \C^\alpha_\Omega \cap \partial B_\delta=\partial^*_{n-1} \Gamma$ up to a $\H^{n-2}$-negligible set, where by $\partial^*_{n-1}\Gamma$ we denote the reduced boundary of $\Gamma$ as a subset of the $(n-1)$-dimensional sphere $\partial B_\delta$. Up to choose any such $\delta$, one can then calculate
\begin{align*}
|\widetilde \C| &= |\C_1| + (1+\eps)^n |\C_2| + \frac \delta n\big((1+\eps)^n-1\big)\H^{n-1}(\Gamma)\,, \\
P(\widetilde\C) &= P(\C_1) + (1+\eps)^{n-1} P(\C_2) + \frac \delta{n-1} \big((1+\eps)^{n-1}-1\big) \H^{n-2}(\partial^*_{n-1}\Gamma)\,.
\end{align*}
As a consequence, a trivial calculation ensures that the $\alpha$-Cheeger ratio of $\widetilde \C$ is strictly better than that of $\C^\alpha_\Omega$ (and this is against the optimality of the latter) as soon as $\eps$ is chosen small enough, if the inequality
\[
(n-1) P(\C_2) + \delta \H^{n-2}(\partial^*_{n-1} \Gamma) < \frac{P(\C^\alpha_\Omega) n|\C_2|}{\alpha|\C^\alpha_\Omega|}
\]
holds. And in turn, since $|\C_2|$ (resp., $P(\C_2)$) is arbitrarily close to $|\C^\alpha_\Omega|$ (resp., to $P(\C^\alpha_\Omega)$) up to have chosen $\delta$ arbitrarily small, the contradiction comes exactly as in~(\ref{asbefore}) if we can select an arbitrarily small $\delta$ such that $\H^{n-2}(\partial^*_{n-1}\Gamma) \delta$ is small as we wish. And finally, this is surely true, because otherwise one would have
\[
\liminf_{\delta\to 0} \H^{n-2}(\partial^*_{n-1}\Gamma) \delta > 0\,,
\]
which would readily imply that $P(\C^\alpha_\Omega)=+\infty$, and this is obviously impossible, being $\C^\alpha_\Omega$ an $\alpha$-Cheeger set.
\\
{\bf 3.,4.,5.} Any $\alpha$-Cheeger set is in particular a minimizer of the perimeter among all the subsets of $\Omega$ having its same volume. Therefore, the regularity properties of the free boundary follow by well-known results, see for instance~\cite{GoMaTa2,LeoPra}, and it only remains to prove the formula for the curvature in the two-dimensional case.\par

To do so we observe that, in the two-dimensional case, a curve with constant curvature is nothing else than an arc of circle, hence the free boundary of an $\alpha$-Cheeger set $\C^\alpha_\Omega$ is a union of arcs of circle. Let us then focus on a single arc; more precisely, let us fix an arc of circle $\gamma$, with radius $r$ and amplitude $2\theta$, which belongs to $\partial\C^\alpha_\Omega$ and which is entirely contained in the interior of $\Omega$; let us also call $P$ and $Q$ the endpoints of $\gamma$. Let us now slightly modify the set $\partial\C^\alpha_\Omega$ as follows: we replace $\gamma$ with a similar arc of circle, still connecting $P$ and $Q$, in such a way that the center of the new arc has been moved of a distance $\eps$ toward the segment $PQ$. Of course, if $|\eps|\ll 1$ then the new set is the boundary of a set $A_\eps$, very similar to $\C^\alpha_\Omega$ and contained inside $\Omega$. A simple trigonometric calculation ensures that the radius and the amplitude of the new arc are given by
\begin{align*}
\tilde r = r -\eps \cos\theta + o(\eps)\,, && \tilde\theta = \theta + \eps\,\frac{\sin\theta}r + o(\eps)\,,
\end{align*}
and as a consequence area and perimeter of $A_\eps$ are
\begin{align*}
|A_\eps| = |\C^\alpha_\Omega| + 2 r \eps(\sin \theta - \theta \cos \theta) + o(\eps)\,, &&
P(A_\eps) = P(\C^\alpha_\Omega) + 2\eps (\sin \theta - \theta \cos \theta)+o(\eps)\,,
\end{align*}
which gives
\[
\frac{P(A_\eps)}{|A_\eps|^{1/\alpha}} = \frac{P(\C^\alpha_\Omega)}{|\C^\alpha_\Omega|^{1/\alpha}} + 
2\eps \,\frac{(\sin\theta-\theta\cos\theta)}{|\C^\alpha_\Omega|^{1/\alpha}}\, \bigg( 1 - \frac{P(\C^\alpha_\Omega)r}{\alpha|\C^\alpha_\Omega|}\bigg)+o(\eps)\,.
\]
Since the $\alpha$-Cheeger ratio of $A_\eps$ cannot be better than that of the $\alpha$-Cheeger set $\C^\alpha_\Omega$, and since any small positive or negative $\eps$ can be chosen, we obtain the equality $P(\C^\alpha_\Omega) r =\alpha |\C^\alpha_\Omega|$, from which the formula~(\ref{alphacurvature}) for the curvature follows.\\
{\bf 6.}
Suppose that an $\alpha$-Cheeger set $\C^\alpha_\Omega$ contains an angle strictly smaller than $\pi$; by ``cutting away'' the corner at a distance $\eps>0$, we can lower the perimeter of a quantity which goes as $\eps$, changing the volume only of a quantity proportional to $\eps^2$. If $\eps\ll 1$, the new set has a strictly better $\alpha$-Cheeger ratio than $\C^\alpha_\Omega$, and since this is impossible we conclude also this point.\\
{\bf 7.}
Let $\C^\alpha_\Omega$ be a non-connected $\alpha$-Cheeger set, and let $\C_1$ and $\C_2$ be two non-empty open sets with disjoint closures such that $\C^\alpha_\Omega=\C_1\cup \C_2$. Since of course $P(\C_1\cup \C_2)= P(\C_1) + P(\C_2)$, while for any $\alpha>1$ it is $|\C_1\cup \C_2|^{1/\alpha} = \big(|\C_1|+ |\C_2|\big)^{1/\alpha}<|\C_1|^{1/\alpha} + |\C_2|^{1/\alpha}$, one has
\[
\min\bigg\{ \frac{P(\C_1)}{|\C_1|^{1/\alpha}},\frac{P(\C_2)}{|\C_2|^{1/\alpha}}\bigg\}
\leq \frac{P(\C_1)+P(\C_2)}{|\C_1|^{1/\alpha}+|\C_2|^{1/\alpha}}
< \frac{P(\C_1\cup \C_2)}{|\C_1\cup \C_2|^{1/\alpha}}
=\frac{P(\C^\alpha_\Omega)}{|\C^\alpha_\Omega|^{1/\alpha}}\,,
\]
Since the above inequality says that there is some set which has $\alpha$-Cheeger ratio strictly better than $\C^\alpha_\Omega$, the contradiction shows this point.\\
{\bf 8.}
In the two-dimensional case, the convex hull of any set $E$ has bigger volume and smaller perimeter than $E$, with strict inequalities if $E$ is not already convex. The convexity of any $\alpha$-Cheeger set corresponding to a convex set $\Omega$ is then obvious. It is to be mentioned that, for the standard Cheeger problem, the convexity of the Cheeger sets of convex domains is known also in the higher dimensional case, see~\cite{StrZie,AltCas,KawLR}.\\
{\bf 9.}
This is obvious, since if $\Omega_1\subseteq\Omega_2$ then any subset of $\Omega_1$ is also a subset of $\Omega_2$. The fact that the strict inclusion does not imply the strict inequality follows by trivial counterexamples.\\
{\bf 10.} Let us take any set $E\subseteq \Omega$ with $|E|\leq 1$, and let $B_E\subseteq B_1\subseteq \Omega$ be a ball with the same volume as $E$. Then,
\[
\frac{P(E)}{|E|^{1/\alpha}} \geq \frac{P(B_E)}{|B_E|^{1/\alpha}} \geq \frac{P(B_1)}{|B_1|^{1/\alpha}}\,,
\]
where the last inequality holds true whenever $\alpha<1^*$. As a consequence, we derive that
\begin{equation}\label{stimora}
h_\alpha (\Omega)= \inf_{E \subseteq \Omega,\, |E|\geq 1} \frac{P(E)}{|E|^{1/\alpha}}\,.
\end{equation}
Observe now that, for any $E\subseteq \Omega$ with $|E|\geq 1$, it is
\[
\frac{P(E)}{|E|^{1/\alpha_1}} \leq \frac{P(E)}{|E|^{1/\alpha_2}}\,,
\]
which by taking the infimum over the sets $E$ and recalling~(\ref{stimora}) concludes the claim.\\
{\bf 11.} The proof of this last claim is almost identical to the previous one: since $|\Omega|\leq 1$, then of course $|E|\leq 1$ for every set $E\subseteq \Omega$, thus
\[
\frac{P(E)}{|E|^{1/\alpha_1}} \geq \frac{P(E)}{|E|^{1/\alpha_2}}
\]
and the claim again follows by taking the infimum over sets $E$.
\end{proof}

\begin{rem}
If an $\alpha_2$-Cheeger set with volume strictly greater than $1$ exists, then point~10 of the previous theorem holds with a strict inequality. If an $\alpha_1$-Cheeger set with volume strictly smaller than $1$ exists, then point~11 of the previous theorem holds with a strict inequality.
\end{rem}

\begin{rem}
As a straight consequence of the previous proof we have obtained that, to compute the $\alpha$-Cheeger constant of any set $\Omega$, one can reduce himself to minimize the $\alpha$-Cheeger ratio among the sets with volume bigger than the biggest ball contained in $\Omega$.
\end{rem}

\section{Existence\label{existence}}

This short section is devoted to show the existence of $\alpha$-Cheeger sets in any given set $\Omega$: the proof is identical to the one for the classical case, and we report it only for the sake of completeness.\par

Before starting, let us briefly compute the $\alpha$-Cheeger ratio of a ball $B_r$ of radius $r>0$, which is
\[
\frac{P(B_r)}{|B_r|^{1/\alpha}} = \frac{n\omega_n r^{n-1}}{(\omega_n r^n)^{\frac 1\alpha}}
= n \omega_n^{1-\frac 1\alpha} r^{n-1-\frac n\alpha}\,.
\]
Notice that the exponent $n-1-\frac n\alpha$ is negative if and only if $\alpha<1^*$, and it is null when $\alpha=1^*$. As a consequence, for $\alpha>1^*$, the $\alpha$-Cheeger ratio of smaller and smaller balls converges to $0$, while for $\alpha=1^*$ all the balls have the same $\alpha$-Cheeger ratio, regardless of their radius. As a consequence, we can observe what follows.
\begin{rem}
If $\alpha>1^*$, then the $\alpha$-Cheeger constant of any set $\Omega$ is $0$, and there are no $\alpha$-Cheeger sets. If $\alpha=1^*$, every set $\Omega$ have the same $\alpha$-Cheeger constant, and the $\alpha$-Cheeger sets are all and only the balls.
\end{rem}
In other words, the $\alpha$-Cheeger problem would be trivial for $\alpha\geq 1^*$, and this is why one always chooses $1<\alpha<1^*$ for the generalized Cheeger problem. Observe that, in principle, the problem would be non-trivial even for $0<\alpha<1$; nevertheless, people usually consider the case $\alpha>1$ because this corresponds to give more importance to the isoperimetric properties of the set; in other words, when $\alpha$ increases then for the subsets $E$ in the definition~(\ref{CheegerPbl}) it becomes more and more important to have a smaller perimeter, instead of having a bigger area.\par

We can now pass to the existence result.

\begin{thm}\label{thmexistence}
Let $\Omega\subseteq\R^n$ be any set with finite volume. Then, for every $1<\alpha<1^*$ there exists an $\alpha$-Cheeger set in $\Omega$; in other words, the infimum in~(\ref{CheegerPbl}) is a minimum.
\end{thm}
\begin{proof}
Let $\{E_k\}_k$ be a minimizing sequence for~(\ref{CheegerPbl}). Then, $\chi_{E_k}$ is a bounded sequence in $BV(\R^n)$ since
\[
\|\chi_{E_k}\|_{BV(\R^n)} = \|\chi_{E_k}\|_{L^1(\R^n)} + \|D \chi_{E_k}\|_{\mathcal M(\R^n)}
=  |E_k| + P(E_k)\,,
\]
and $|E_k|$ is bounded by $|\Omega|$, while $P(E_k)$ is uniformly bounded because $\{E_k\}$ is a minimizing sequence for~(\ref{CheegerPbl}) --we can assume that $h_\alpha(\Omega)<+\infty$, because otherwise any subset of $\Omega$ is an $\alpha$-Cheeger set, and there is nothing to prove. Thanks to the classical compactness results for $BV(\R^n)$ (see for instance~\cite{AmFuPa}), and recalling that $\Omega$ has finite volume, we obtain that, up to a subsequence, $\chi_{E_k}$ weakly converges in $BV$ to some function $\varphi$. In particular, the convergence is strong in $L^1$, and this implies that $\varphi$ is also the characteristic function of some set $E\subseteq \Omega$; moreover, the lower semi-continuity implies that
\[
P(E) = \|D \chi_{E}\|_{\mathcal M(\R^n)} \leq \liminf_{k\to\infty} \|D \chi_{E_k}\|_{\mathcal M(\R^n)}
=\liminf_{k\to\infty} P(E_k)\,.
\]
Summarizing, we have proved the existence of a set $E\subseteq \Omega$ such that
\begin{align}\label{lookhere}
|E| = \lim_{k\to\infty} |E_k|\,, && P(E) \leq \liminf_{k\to \infty} P(E_k)\,.
\end{align}
The proof will be concluded once we show that $E$ is an $\alpha$-Cheeger set in $\Omega$, which in turn is obvious from~(\ref{lookhere}) as soon as we observe that $|E|>0$. In fact, let us assume that $|E|=0$: this implies that $|E_k|\to 0$, hence we would have, calling $r_k$ the radius of a ball $B_k$ with $|B_k|=|E_k|$,
\[
\frac{P(E_k)}{|E_k|^{1/\alpha}} \geq \frac{P(B_k)}{|B_k|^{1/\alpha}} = n \omega_n^{1-\frac 1\alpha} r_k^{n-1-\frac n\alpha} \to \infty\,,
\]
where the last convergence holds since $1<\alpha<1^*$. Since this last estimate is in contradiction with the fact that $\{E_k\}$ is a minimizing sequence for~(\ref{CheegerPbl}), we have shown that it must be $|E|>0$ and, as noticed above, this concludes the proof.
\end{proof}

\section{$\alpha$-Cheeger sets for rectangles\label{rectangles}}

In this section, we study in detail the $\alpha$-Cheeger problem for the case of the rectangles, and in the next section we will generalize our results to the case of the strips (which are ``deformations'' of a rectangle, see Definition~\ref{defstrip}). This is not just an example: indeed, the study of $2$-dimensional strips or $3$-dimensional waveguides is already important for the standard Cheeger problem (we will discuss this in the next section); moreover, for the generalized Cheeger problem, we will observe that rectangles give simple counterexamples to classical facts which are instead true for the standard Cheeger problem.

The plan of this section is the following: first we give a structure result for the $\alpha$-Cheeger sets in rectangles, then we show how this gives the above-mentioned counterexamples, and finally, for the sake of completeness, we explicitely compute the $\alpha$-Cheeger constants of rectangles.

\subsection{A structure theorem for $\alpha$-Cheeger sets of rectangles}

Through all this section, we denote by $R_L=(-L/2,L/2)\times (-1,1)$ the rectangle of length $L\geq 2$ and width $2$, and by $R_\infty=\R\times (-1,1)$ the unbounded rectangle. Of course, by trivial rescaling one can treat also any other rectangle. Let us recall the results which are known for the rectangles in the standard case.

\begin{thm}\label{stand}
The infinite rectangle $R_\infty$ admits no Cheeger sets, but the rectangles $R_L\subseteq R_\infty$ are a mimimizing sequence for~(\ref{chst}) when $L\to \infty$, and $h(R_\infty)=1$. Any rectangle $R_L$ admits a unique Cheeger set, which is the union of all the balls of radius $1/h(R_L)<1$, according to point~8 of Theorem~\ref{classical}: in particular, this set is the whole rectangle with the four corners ``cut away'' by four arcs of circle.
\end{thm}

We will notice that the situation is quite different when $\alpha>1$: we will observe in Lemma~\ref{primo} that the diameter of any $\alpha$-Cheeget set is bounded independently of $L$, and we will completely characterize, in Theorem~\ref{thm: structure theorem for rectangles}, the structure of the $\alpha$-Cheeger sets, in particular showing that there exist $\alpha$-Cheeger sets also for the unbounded rectangle $R_\infty$. In the rest of the section, for simplicity of notation, we will write $\C^\alpha_L$ to denote an $\alpha$-Cheeger set in $R_L$, instead of $\C^\alpha_{R_L}$. First of all, we observe that any $\alpha$-Cheeger set of a rectangle can only have two possible shapes, depicted in Figure~\ref{Fig:rectangles}.

\begin{lem}\label{zeresimo}
Let $\C^\alpha_L$ be an $\alpha$-Cheeger set in the rectangle $R_L$. Then, the free boundary of $\C^\alpha_L$ is not empty and it has constant curvature $1/r$. Moreover, either $r<1$ and the free boundary is given by four arcs of circle of amplitude $\pi/2$ ``cutting away'' the four corners of $R_L$, as in Figure~\ref{Fig:rectangles}~(left), or $r=1$ and the free boundary is given by two arcs of circle of amplitude $\pi$ connecting the top and the bottom sides of $R_L$, as in Figure~\ref{Fig:rectangles}~(right).
\end{lem}
\begin{proof}
Thanks to Theorem~\ref{properties}, point~6, we know that (the closure of) $\C^\alpha_L$ cannot contain any of the four corners of $R_L$. As a consequence, the free boundary of $\C^\alpha_L$ is not empty, and by point~4 of Theorem~\ref{properties} we know that it has constant curvature, call it $1/r$. Thus, the free boundary of $\C^\alpha_L$ is given by arcs of circle of radius $r$, connecting regular points of the boundary of the rectangle. Since by point~5 of Theorem~\ref{properties} these arcs must meet $\partial R_L$ tangentially, the two endpoints of each arc must lie on different sides of the rectangle.\par
\begin{figure}[htbp]
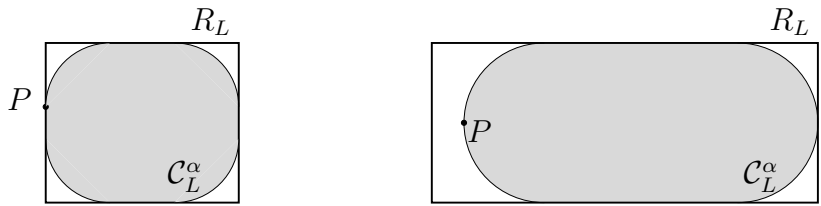
\caption{The two possible shapes of the $\alpha$-Cheeger sets in rectangles.}\label{Fig:rectangles}
\end{figure}
Consider now the left side of $R_L$; suppose first that its intersection with $\C^\alpha_L$ is not empty, and call $P$ the highest point of this intersection; keep in mind that $P$ cannot be a corner of the rectangle. Thus, there is a connected component of the free boundary which is an arc of circle of radius $r$ starting from $P$; as a consequence, the distance of $P$ from the upper-left corner of the rectangle is at least $r$. Since, analogously, the lowest point in the intersection of $\C^\alpha_L$ with the left side of $R_L$ has at least distance $r$ from the lower-left corner, we immediately derive that $r\leq 1$. From that, and keeping in mind that $L\geq 2$, it follows that the arc starting at $P$ cannot reach the right side of the rectangle, and it must instead reach the upper side (in principle, it could also reach directly the lower side, but this would readily imply that $\C^\alpha_L$ is a disk with radius $r<1$, and this case can be immediately ruled out by observing that $R_L$ contains also disks of radius $1$, whose $\alpha$-Cheeger ratio is strictly better). Therefore, we deduce that the distance of $P$ from the upper-left corner is \emph{exactly} $r$. Summarizing, we have proved that if $\C^\alpha_L$ has a non-empty intersection with the left side of the rectangle, then the free boundary of $\C^\alpha_L$ contains the two arcs of circle of radius $r$ and amplitude $\pi/2$ connecting the four points having distance $r$ from the upper-left and the lower-left corner.\par

Consider now the case when $\C^\alpha_L$ does not intersect the left side of the rectangle; then, let us call $P$ a point of $\partial \C^\alpha_L$ having minimal distance to the left side. Since $\partial \C^\alpha_L$ has positive curvature, $P$ must be in the interior of the rectangle, hence it belongs to some arc of circle with radius $r$; by assumption, this arc of circle does not touch the left side of $R_L$, so it must connect the upper and the lower side of the rectangle (keep in mind the first case considered above, to exclude that this arc reaches the right side). Since the arcs meet $\partial R_L$ tangentially, we deduce that $r=1$ necessarily, so the arc containing $P$ has amplitude $\pi$. Summarizing, we have proved that if $\C^\alpha_L$ does not intersect the left side of the rectangle, then $r=1$ and the closest point of $\partial\C^\alpha_L$ to the left side is contained in an arc of circle, belonging to the free boundary of $\C^\alpha_L$, of radius $r$ and amplitude $\pi$, which connects the two horizontal sides of $R_L$.\par

Putting together the two cases discussed above, and their obvious counterparts for the right sides, we immediately realize that there are only two possibilities: either $r<1$, and then the boundary of $\C^\alpha_L$ contains the four arcs of radius $r$ and amplitude $\pi/2$ connecting the eight points at distance $r$ from the corners, or $r=1$ and the boundary of $\C^\alpha_L$ contains two arcs of radius $r=1$ and amplitude $\pi$ connecting the above and the bottom side of the rectangle. Notice that, in this case, it is possible that these two arcs meet also the lateral sides of the rectangle, but this is not necessary: for instance, in Figure~\ref{Fig:rectangles}~(right), the left arc does not touch the left side, while the right arc touches the right side.\par

To conclude the proof, we have then to prove that the free boundary of $\C^\alpha_L$ cannot contain any other arc, except those described above. Indeed, if $r<1$ there cannot be any other arc, simply because there does not exist any other arc of radius $r$ which connects two points in the boundary of the rectangle reaching them in a tangential way; instead, if $r=1$, the presence of other arcs in the free boundary would imply that $\C^\alpha_L$ is not connected, against point~7 of Theorem~\ref{properties}.
\end{proof}

\begin{cor}\label{quinto}
Let $\C^\alpha_L$ be an $\alpha$-Cheeger set in the rectangle $R_L$ with radius of curvature $r=1$. Then $\C^\alpha_L$ is a rectangle of height $1$ and width $M$ topped with two half-disks of radius $1$, where
\begin{equation}\label{eq: value of M}
M=M(\alpha) := \frac{\pi}{2} \cdot \frac{2-\alpha}{\alpha-1}\,.
\end{equation}
\end{cor}
\begin{proof}
By Lemma~\ref{zeresimo} we know the two possible shapes of an $\alpha$-Cheeger set in $R_L$; if $r=1$, then the shape must be the one depicted in Figure~\ref{Fig:rectangles}~(right), so actually $\C^\alpha_L$ is a rectangle topped with two half-disks, and to conclude we only have to evaluate $M$. To do so, by simply rewriting~(\ref{alphacurvature}) we can express the radius of curvature of the free boundary in terms of the $\alpha$-Cheeger constant and the area as
\begin{equation}\label{eq:raggioperalfa}
r =\frac{\alpha}{h_\alpha(R_L)}\,|\C^\alpha_L|^{1-\frac 1\alpha}\,,
\end{equation}
and since $r=1$ we derive
\[
\alpha|\C^\alpha_L|^{1-\frac 1\alpha}=h_\alpha(R_L) = \frac{P(\C^\alpha_L)}{|\C^\alpha_L|^{1/\alpha}}\,.
\]
Finally, substituting the values of area and perimeter of $\C^\alpha_L$ gives~(\ref{eq: value of M}).
\end{proof}

Let us now show that the diameter of the $\alpha$-Cheeger sets in the rectangles is uniformly bounded.

\begin{lem}\label{primo}
For every $\alpha\in (1,2)$, there exists $\overline d=\overline d(\alpha)$ such that the diameter of any $\alpha$-Cheeger set $\C^\alpha_L$ of $R_L$ is less than $\overline d$. More precisely, the $\alpha$-Cheeger ratio of sets with diameter $d$ diverges when $d\to\infty$.
\end{lem}
\begin{proof}
For any $L<\infty$, we already know by Theorem~\ref{thmexistence} that there exists an $\alpha$-Cheeger set $\C^\alpha_L$, and this set is connected by Theorem~\ref{properties}. Then, calling $d$ the diameter of this set, we know that $P(\C^\alpha_L)\geq 2d$, and on the other hand $|\C^\alpha_L|\leq 2d$ because the width of the rectangle is $2$. Hence, we can estimate the $\alpha$-Cheeger constant of $\C^\alpha_L$ as
\[
h_\alpha(R_L) = \frac{P(\C^\alpha_L)}{|\C^\alpha_L|^{1/\alpha}} \geq (2d)^{1-\frac 1\alpha}\,.
\]
Since the latter quantity diverges for $d\to\infty$, and on the other hand the $\alpha$-Cheeger constants of the rectangles $R_L$ are uniformly bounded for $L$ large thanks to point~9 of Theorem~\ref{properties}, we obtain the thesis.
\end{proof}

An easy consequence of the above result is that the $\alpha$-Cheeger sets in $R_L$ necessarily have the shape of Figure~\ref{Fig:rectangles}~(right) as soon as $L$ is big enough.
\begin{cor}\label{quarto}
For every $\alpha \in (1,2)$, if $L$ is big enough then the radius of curvature is $r=1$.
\end{cor}
\begin{proof}
Thanks to Lemma~\ref{zeresimo}, we know that either $r=1$, or the diameter of an $\alpha$-Cheeger set is at least $L$. Since this is impossible for $L>\overline d(\alpha)$ according to Lemma~\ref{primo}, we conclude.
\end{proof}

To get a complete characterization of the $\alpha$-Cheeger sets in the rectangles, we have to distinguish for which values of $L$ the shape of the $\alpha$-Cheeger sets is the first or the second possible one, according to Lemma~\ref{zeresimo}. Keeping in mind Corollary~\ref{quinto}, for sure the possibility $r=1$ can only happen if $L\geq M(\alpha)+2$, because otherwise a rectangle of length $M$ topped with two half-circles cannot fit into $R_L$. Actually, we will see that this condition is also sufficient.

\begin{thm}[Structure theorem of $\alpha$-Cheeger sets for rectangles]\label{thm: structure theorem for rectangles}
For any $\alpha \in (1,2)$ the following holds:
\begin{itemize}
\item[(i)]if $L < M(\alpha) + 2$, then $R_L$ has a unique $\alpha$-Cheeger set, obtained from the whole $R_L$ by cutting away the corners with four arcs of radius $r$, being
\begin{equation}\label{eq: raggio per rettangolo corto}
r = \frac{L+2 - \sqrt{(L+2)^2-2(4-\pi)(2-\alpha)L\alpha}}{(4-\pi)(2-\alpha)} < 1\,;
\end{equation}
\item[(ii)]if $L= M(\alpha)+2$, then $R_L$ has a unique $\alpha$-Cheeger set, obtained from the whole $R_L$ by cutting away the corners with four arcs of radius $r=1$;
\item[(iii)]if $L> M(\alpha)+2$, then $R_L$ has not a unique $\alpha$-Cheeger set; more precisely, its $\alpha$-Cheeger sets are all the rectangles of sides $M(\alpha)$ and $2$ topped by two half-disks of radius $1$ which fit into $R_L$.
\end{itemize}
\end{thm}
\begin{proof}
As noticed above, if $L<M(\alpha)+2$ then it is necessarily $r<1$, and thus by Lemma~\ref{zeresimo} an $\alpha$-Cheeger set must be as in Figure~\ref{rectangles}~(left), that is, the whole rectangle with the four corners cut away with arcs of radius $r$. Hence, an $\alpha$-Cheeger set is completely characterized by the value of $r$: notice that, in principle, there could be more $\alpha$-Cheeger sets, corresponding to different values of $r$. For any $0<t<1$, let us now call $\RR_t$ the rectangle $R_L$ with the four corners cut away with quarters of circle of radius $t$: by construction, the $\alpha$-Cheeger sets are exactly the sets $\RR_t$ minimizing the $\alpha$-Cheeger ratio. Since the perimeter and area of $\RR_t$ can be calculated as
\begin{align}\label{formulaperimearea}
P(\RR_t) = 2L +4 -(8-2\pi)t  \,, && |\RR_t| = 2L - (4-\pi)t^2  \,,
\end{align}
a straightforward minimization argument shows that there is a unique $0<t<1$ minimizing the $\alpha$-Cheeger ratio, and it is given by formula~(\ref{eq: raggio per rettangolo corto}). Hence, point~(i) is concluded. One can check that $r<1$ for every $L<M(\alpha)+2$, and $r\to 1$ for $L \to M(\alpha)+2$.\par

Let us now consider the case $L=M(\alpha)+2$. The same calculations as before ensure that, among all the sets $\RR_t$, the one corresponding to $t=1$ uniquely minimizes the $\alpha$-Cheeger ratio. As a consequence, it is impossible that $r<1$ and then it must be $r=1$. By Corollary~\ref{quinto}, we know that all the $\alpha$-Cheeger sets must be rectangles with height $1$ and width $M$ topped by two half-disks of radius $1$. However, since $L=M(\alpha)+2$, there is only one such set in $R_L$, whose free boundary is tangent to both the lateral sides of the rectangle. Hence, also point~(ii) is concluded.\par

Let us finally consider point~(iii): we will have obtained it, as soon as we exclude that for some $L>M(\alpha)+2$ there is an $\alpha$-Cheeger set with radius of curvature $r<1$. Suppose then, by contradiction, the existence of some $L>M+2$ admitting an $\alpha$-Cheeger set $\C^\alpha_L$ with $r<1$. Let us now take some $L'$ much bigger than $L$: the set $\C^\alpha_L$ is also contained in $R_{L'}$ but, according to Corollary~\ref{quarto}, any $\alpha$-Cheeger set in $R_{L'}$ must have radius of curvature equal to $1$, so $\C^\alpha_L$ is \emph{not} an $\alpha$-Cheeger set in $R_{L'}$. By Corollary~\ref{quinto}, moreover, the $\alpha$-Cheeger sets in $R_{L'}$ are the rectangles of width $M$ topped with half-disks of radius $1$; observe that, since $L>M+2$, there is such a set, call it $\C^\alpha_{L'}$, which is contained also in $R_L$. This immediately gives a contradiction: indeed, the $\alpha$-Cheeger ratio of $\C^\alpha_{L'}$ is strictly better than that of $\C^\alpha_L$, because $\C^\alpha_L$ is not an $\alpha$-Cheeger set of $R_{L'}$; and since both $\C^\alpha_L$ and $\C^\alpha_{L'}$ are contained in $R_L$, this implies that $\C^\alpha_L$ cannot be an $\alpha$-Cheeger set in $R_L$. The proof is then concluded.
\end{proof}

\begin{rem}
By sending $\alpha\to 1$ or $\alpha\to 2$ in the previous result one derives, of course, the already known results for the standard case (when $\alpha=1$) and the trivial results of the case $\alpha=1^*=2$, already discussed at the beginning of Section~\ref{existence}. In particular, our equation~(\ref{eq: raggio per rettangolo corto}) extends the corresponding equation~(11) of~\cite{KawLR}.
\end{rem}

\subsection{Counterexamples given by long and thin rectangles}

Here we briefly discuss how the situation for the rectangles has completely changed from the standard case to the generalized one. In particular, all the claims of Theorem~\ref{stand} fail, at least when $L>M(\alpha)+2$. More in general, we are going to see that sufficiently long rectangles give counterexamples to all the properties of Theorem~\ref{classical} whose analogue is not stated in Theorem~\ref{properties}.\par

First of all, the infinite rectangle $R_\infty$ \emph{does} admit $\alpha$-Cheeger sets, namely, any rectangle of length $M(\alpha)$ topped with two half-disks. In particular, the rectangles $R_L$ are \emph{not} a minimizing sequence for~(\ref{CheegerPbl}), since their $\alpha$-Cheeger constants explode when $L\to \infty$.\par

As soon as $L>M(\alpha)+2$, $R_L$ does \emph{not} admit a unique $\alpha$-Cheeger set, and point~8 of Theorem~\ref{classical} almost completely fails for $\alpha>1$ (notice the difference with point~8 of Theorem~\ref{properties}): it is true that any $\alpha$-Cheeger set is convex, but it is not unique, and in particular it is \emph{not} the union of all the balls with radius $r$.\par

Another property which is easily observed for the rectangles in the standard case is the following: if $L_1<L_2$, then the Cheeger set corresponding to $R_{L_1}$ is contained in the Cheeger set corresponding to $R_{L_2}$. For the $\alpha$-Cheeger problem, this is only partially true: more precisely, any $\alpha$-Cheeger set in $R_{L_1}$ is contained in some $\alpha$-Cheeger set in $R_{L_2}$, but there are $\alpha$-Cheeger sets in $R_{L_2}$ which do not contain any $\alpha$-Cheeger set in $R_{L_1}$. Concerning different values of $\alpha$, the same happens: again by the non-uniqueness and the possible translations, in $R_L$ there are $\alpha_1$-Cheeger sets and $\alpha_2$-Cheeger sets which are not contained one into the other.

\subsection{Computation of $h_\alpha$ for a given rectangle}

In this last short subsection, for the sake of completeness, we briefly compute the $\alpha$-Cheeger constant of some given rectangle $R_L$ depending on the power $\alpha\in (1,2)$; we can assume without loss of generality that $L\geq 2$ (otherwise, it is enough to rotate and rescale the rectangle). Of course, first of all we need to determine when the situation is the one of case~(i), or~(ii), or~(iii) of the structure Theorem~\ref{thm: structure theorem for rectangles}. Actually, we already know that everything depends on the fact that $L$ is bigger or smaller than $M(\alpha)+2$, and recalling~(\ref{eq: value of M}) this is equivalent to $\alpha$ being bigger or smaller than
\[
\bar \alpha(L) = \frac{2(\pi+L-2)}{\pi+2L-4}\,.
\]
It is immediate to observe that $L\mapsto \bar\alpha(L)$ is a strictly decreasing function, and that $\bar\alpha(L) \in (1,2]$, being $\bar\alpha(2)=2$ and $\bar\alpha(+\infty)=1$.
\begin{prop}
The $\alpha$-Cheeger constant of $R_L$ is given by
\begin{equation}\label{eq: h per alpha over the lower bound}
h_\alpha (R_L) =\left\{\begin{array}{ll}
\begin{aligned}\alpha \bigg ( \frac{\pi}{\alpha-1} \bigg)^{1 - \frac{1}{\alpha}}\end{aligned} & \hbox{if $\alpha \geq \bar{\alpha}(L)$}\,,\\[15pt]
\begin{aligned} \frac{2L+4-(8-2\pi)r}{\big(2L-(4-\pi)r^2\big)^{1/\alpha}} \end{aligned}
\qquad\qquad& \hbox{otherwise}\,,
\end{array}\right.
\end{equation}
where $r$ is given by~(\ref{eq: raggio per rettangolo corto}).
\end{prop}
\begin{proof}
If $\alpha\geq \bar\alpha(L)$, then $L\geq M(\alpha)+2$ and the $\alpha$-Cheeger set is a rectangle of length $M(\alpha)$ plus two half-disks. As a consequence, by~(\ref{eq:raggioperalfa}) together with the fact that $r=1$ and using~(\ref{eq: value of M}), we get
\[
h_\alpha(R_L)=\alpha |\C^\alpha_L|^{1-\frac 1\alpha}
= \alpha \big(2M(\alpha)+\pi\big)^{1-\frac 1\alpha}
= \alpha \bigg(\frac \pi{\alpha-1}\bigg)^{1-\frac 1\alpha}\,,
\]
according with~(\ref{eq: h per alpha over the lower bound}).\par
If $\alpha<\bar\alpha(L)$, instead, we know precisely the shape of the $\alpha$-Cheeger set of $R_L$, which is the whole rectangle with the four corners cut away with arcs of circle of radius $r$, where $r$ is given by~(\ref{eq: raggio per rettangolo corto}). Then, formula~(\ref{eq: h per alpha over the lower bound}) follows straightforwardly.
\end{proof}

\begin{rem}
By sending $L\to \infty$, the first line of formula~(\ref{eq: h per alpha over the lower bound}) gives also the value of the $\alpha$-Cheeger constant $h_\alpha$ for the unbounded rectangle $R_\infty$. If we then send $\alpha\searrow 1$, we obtain that $h_\alpha(R_\infty)$ converges to $1$, which agrees with the fact that $h_1(R_\infty)=1$, as well-known.
\end{rem}

\section{$\alpha$-Cheeger sets for strips\label{strips}}

Lately, due to their importance in many practical applications, especially in engineering and in medicine, a vivid interest has arisen around the waveguides, and their $2$-dimensional counterpart, the strips. Roughly speaking, the waveguides are deformations of cylinders, while strips are deformation of rectangles. The main reason of the interest about these sets in the applications is due to their interesting optical properties; mathematically speaking, the main feature that one is interested in, is the study of the behaviour of the eigenvalues, which in turn is strictly related with the study of the Cheeger constant. More about the meaning and the importance of waveguides and strips can be found in~\cite{DucExn,KreKri} and the references therein, while recent mathematical studies can be found in~\cite{KrePra,LeoPra}. Actually, for the applications the main case is that of ``long and thin'' waveguides or strips.\par

The aim of this short section is to show that the results that we found for rectangles in the last sections can be easily extended to the case of the strips. Let us start with the relevant definitions.

\begin{defin}[Strips]\label{defstrip}
Let $I\subseteq\R$ be an interval, let $\gamma:I\to\R^2$ be a ${\rm C}^{1,1}$ curve, parametrized at unit speed, and let us denote by $\nu: I \to \R^2$ the unit normal vector to the curve $\gamma$. Define the map $\Psi: I \times [-1,1]\to\R^2$ as
\[
\Psi(s,t) = \gamma(s) + t \nu(s)\,.
\]
If the map $\Psi$ is injective, then we say that the image of $\Psi$ is a \emph{strip of half-width $1$}. We call \emph{length of the strip} the length of the interval $I$; moreover, calling $I=(a,b)$, we will refer to $\Psi(a,\cdot)$ and $\Psi(b,\cdot)$ as the \emph{lateral sides} of the strip, as well as to $\Psi(\cdot,\pm 1)$ as the \emph{horizontal sides} of the strip.
\end{defin}

Notice that the injectivity of the map $\Psi$ in particular implies that the modulus of the curvature of $\gamma$ is bounded by $1$; of course, a rectangle is the particular case of a strip when $\gamma$ is the parametrization of a segment. Notice that, depending on the different properties of the curve $\gamma$, all the possible strips can be of four different kinds:
\begin{itemize}
\item if $I = [0,L)$ and $\gamma (0) = \gamma (L)$, then we have a ``closed strip'', usually called \emph{generalized annulus};
\item if $I = (0,L)$ and $\gamma (0) \neq \gamma (L)$, then we have an open \emph{finite strip}; 
\item if $I = (0, +\infty)$, then we have an open \emph{semi-infinite strip};
\item if $I = (-\infty, +\infty)$, then we have an open \emph{infinite strip}.
\end{itemize}

In the two latter cases, we will denote by $S_P$ the bounded strip corresponding to the restriction of $\gamma$, respectively, to $I_P = (0,P)$ and to $I_P = (-P/2, P/2)$. The following results, for the standard Cheeger problem, were proved in~\cite{KrePra,LeoPra}:
\begin{itemize}
\item A generalized annulus is always the unique Cheeger set in itself;
\item A finite strip with $L\geq 9\pi/2$ has always a unique Cheeger set, corresponding to the whole strip with the four corners cut away by arcs of circle;
\item A semi-infinite strip and an infinite strip never have a Cheeger set, but the subsets $S_P$ are always minimizing sequences when $P\to \infty$.
\end{itemize}
In particular, for the annuli and for the finite strips the two properties, which are characteristic of convex sets, are still valid, namely, the uniqueness of the Cheeger set, and the fact that it is the union of all the balls of the correct radius $r=1/h_1(\Omega)$. Roughly speaking, these results ensure that for strips and generalized annuli more or less the same properties as for rectangles and annuli are valid, in the standard Cheeger problem. We can now show that the same is true for the $\alpha$-Cheeger problem.\par

Indeed, let us consider a finite strip $\Omega$ with length $L\geq 9\pi/2$ (this bound on the length is not really needed, it just makes things simpler to treat, see the comment in Remark~\ref{bc}). Let us discuss the possible shapes of the $\alpha$-Cheeger sets. First of all, we know that an $\alpha$-Cheeger set $\C^\alpha_\Omega$ exists, and it cannot be the whole strip because the $\alpha$-Cheeger set cannot contain corners; as a consequence, the intersection of $\partial\C^\alpha_\Omega$ with $\Omega$ is a non-empty union of arcs of circle having radius $r$. Let us now argue as in Lemma~\ref{zeresimo}: by the bound on $L$, an arc cannot join the left and the right sides of $\Omega$, and by the curvature bound and the tangential property given by point~(5) of Theorem~\ref{properties}, an arc cannot joint two points on a same side of $\Omega$. Therefore, every arc should join the two horizontal sides, or an horizontal and a vertical side. Moreover, an arc can join the two horizontal sides only if $r=1$. As a consequence, let us consider separately the two possibilities.\par

If $r<1$, then there is only a possible shape for the $\alpha$-Cheeger set, which is the whole strip with the four corners cut away with arcs of radius $r$. Notice, however, that formulas~(\ref{formulaperimearea}) for the area and perimeter of the whole strip without the four corners cut away are no more valid in general, then also formula~(\ref{eq: raggio per rettangolo corto}) is in general false, and the exact value of $r$ depends not only on $L$ but also on $\gamma$.\par

Suppose, instead, that $r=1$: we claim that, in this case, there are two arcs with amplitude $\pi$ joining the upper and the lower side of $\Omega$, possibly (but not necessarily) also touching the lateral sides, as in Figure~\ref{Fig:rectangles}. Indeed, again as in Lemma~\ref{zeresimo}, if no arc is touching the left side of the strip, then of course the ``left part'' of the free boundary must be an arc between the upper and the lower side. On the other hand, assume that some arc joins the left and the upper sides of the strip, and call $Q$ the center of this arc; then, the distance of $Q$ from the upper side of the strip equals $r=1$, and so the distance of $Q$ from the lower side of the strip is also $1$. As a consequence, the arc of circle which rules out the lower-left corner of the strip, which does not belong to the $\alpha$-Cheeger set, must necessarily be also centered in the same point $Q$: the two arcs, then, are actually a single arc with amplitude $\pi$.\par

Summarizing, we have observed that also for the finite strips there can only be two possible shapes for the $\alpha$-Cheeger sets, namely, the whole strip with the four corners cut away with arcs of radius $r<1$, or some smaller strip of length $M<L$ contained in $\Omega$ topped with two half-disks (we will call these sets ``topped substrips''). In order to distinguish between the two possibilities, we have to calculate area and perimeter of the topped substrips: in fact, a simple calculation (for instance done in~\cite[Proposition~3.5]{LeoPra}) ensures that \emph{all} the topped substrips of width $M$ have the same area and perimeter, regardless of their shape, namely $2M+\pi$ and $2M+2\pi$ (as in the case of the rectangle, which is a particular strip). As a consequence, the same calculation as in Corollary~\ref{quinto} ensures that a topped substrip can be an $\alpha$-Cheeger set only if $M$ is given by~(\ref{eq: value of M}). In addition, the very same argument as in the proof of Theorem~\ref{thm: structure theorem for rectangles} still ensures that, if a topped substrip with the correct value of $M$ fits into $\Omega$, then it must necessarily be an $\alpha$-Cheeger set for $\Omega$. In other words, the $\alpha$-Cheeger sets are all the topped substrips of width $M(\alpha)$, if there are any, and otherwise the strips with the corners cut away.\par

Finally, let us discuss the existence of topped substrips of width $M(\alpha)$ fitting into the strip $\Omega$: in the case of rectangles, of course this happened if and only if $L\geq M(\alpha)+2$, but for the general strips the situation is different. Of course there can be no topped substrip if $L<M(\alpha)+2$, but it is also possible that a strip has length $L>M(\alpha)+2$ and no topped substrips with width $M(\alpha)$; one can easily notice that such a topped substrip exists for sure only when $L\geq M(\alpha)+\pi$, and this bound is sharp and corresponds to a curve $\gamma$ which has exactly curvature $1$ around its extremes. Summarizing, we have proved the following result.
\begin{thm}[Structure theorem of $\alpha$-Cheeger sets for open strips]
Let $\Omega$ be an open strip of length $L\geq 9\pi/2$ (with $L=+\infty$ for semi-infinite or infinite strips). Then, for any $\alpha\in (1,2)$ the following hold:
\begin{itemize}
\item[(i)]{if $L < M(\alpha) + 2$ then $\Omega$ has a unique $\alpha$-Cheeger set found by cutting away the corners of $\Omega$ with arcs of radius $r<1$, where $r$ satisfies~(\ref{eq:raggioperalfa}), but not necessarily~(\ref{eq: raggio per rettangolo corto});}
\item[(ii)]{if $L\geq M(\alpha) + \pi$ then $\Omega$ has not a unique $\alpha$-Cheeger set; more precisely, its Cheeger sets are all the topped substrips of length $M(\alpha)$ which fit into $\Omega$;}
\item[(iii)]{If $M(\alpha)+2\leq L < M(\alpha)+\pi$, then the $\alpha$-Cheeger sets of $\Omega$ are all the topped substrips of length $M(\alpha)$ which fit into $\Omega$, if there is any such set; otherwise, there is a unique Cheeger set, which is as in case~{\rm (i)}.}
\end{itemize}
\end{thm}

To conclude, let us consider the case of a generalized annulus. There are again two possibilities for an $\alpha$-Cheeger set $\C^\alpha_\Omega$: either it coincides with $\Omega$, thus the whole annulus is the $\alpha$-Cheeger set of itself, or the boundary $\partial\C^\alpha_\Omega$ intersects the interior of $\Omega$, and thus it is a finite union of arcs with radius $r\leq 1$. However, for an annulus $\Omega$, the boundary of $\Omega$ is done by two disconnected curves (the ``internal'' and the ``external'' boundary), and an immediate geometric argument ensures that an arc of circle with radius $r< 1$ cannot connect two disctinct points of the same connected component. As a consequence, it must be $r=1$, and thus an $\alpha$-Cheeger set not coinciding with $\Omega$ is again some topped substrip of length $M(\alpha)$; as before, all such sets have the same perimeter and area, so they are all $\alpha$-Cheeger sets (or none of them is so). Moreover, this time it is no more true that such a topped substrip which fits into $\Omega$ is automatically an $\alpha$-Cheeger set: indeed, if the two half-circles are almost tangent, then an immediate calculation ensures that the whole annulus has a strictly better $\alpha$-Cheeger ratio. Actually, since every topped substrip has area $2M(\alpha)+\pi$ and perimeter $2M(\alpha)+2\pi$, while the whole annulus has both perimeter and area equal to $2L$, it is clear that the topped substrips are better than the whole annulus (thus, they are $\alpha$-Cheeger sets) if and only if
\begin{equation}\label{eq: bound annulus}
\frac{2\pi +2M(\alpha)}{(2M(\alpha)+ \pi)^{1/\alpha}} \leq \frac{2L}{(2L)^{1/\alpha}}\,.
\end{equation}
We can then conclude with the structure result for the generalized annuli.
\begin{thm}[Structure theorem of $\alpha$-Cheeger sets for generalized annuli]
Let $\Omega$ be a generalized annulus of length $L\geq 9\pi/2$. Then, for any $\alpha \in (1,2)$ the following hold:
\begin{itemize}
\item[(i)]{if there are no topped substrips of length $M(\alpha)$ which fit in $\Omega$ (in particular, this always happens if $L< M(\alpha)+2$), or if such sets exist but the opposite inequality in~(\ref{eq: bound annulus}) holds true, then the only $\alpha$-Cheeger set in $\Omega$ is the whole $\Omega$ itself;}
\item[(ii)]{if topped substrips of length $M(\alpha)$ which fit in $\Omega$ exist and the strict inequality in~(\ref{eq: bound annulus}) holds, then such sets are all and only the $\alpha$-Cheeger sets in $\Omega$;}
\item[(iii)]{if topped substrips of length $M(\alpha)$ which fit in $\Omega$ exist and equality in~(\ref{eq: bound annulus}) holds, then such sets, as well as the whole $\Omega$ itself, are all and only the $\alpha$-Cheeger sets in $\Omega$.}
\end{itemize}
\end{thm}

\begin{rem}\label{bc}
Let us briefly comment about the bound of the length $L\geq 9\pi/2$. Exactly as in~\cite{LeoPra}, the meaning of this bound is only to ensure that the strip has actually four corners, as well as to exclude that an arc of circle in the free boundary can connect the left and the right side of the boundary. Indeed, if we consider for instance a curve $\gamma$ given by an arc of circle of radius $1$ and angle $\pi/2$, then the corresponding strip is a quarter of a disk, which has only three corners instead of four. As a consequence, if one wants to consider all the strips without a lower bound on the length, then some simple but boring modifications in the claims and in the proofs are needed. Since the claims are already technical enough, we preferred to consider only the case of lengths bounded from below, also because, as explained above, the whole interest in the study of the strips comes for extremely long ones. For the same reason, we did not try to optimize the lower bound on the length, the number $9\pi/2$ is probably too high but we did not find it interesting to obtain a better constant.
\end{rem}

\begin{rem}
Notice that, in the limit cases when $\alpha=1$ or $\alpha=2$, equation~(\ref{eq: bound annulus}) is respectively never verified or always verified. As a consequence, for the standard Cheeger problem we have found the already known result cited above, that is, the whole generalized annulus is always the unique Cheeger set of itself; on the other hand, for the limit case $\alpha=2$, we deduce the fact that the whole annulus is never $\alpha$-Cheeger in itself (which is obvious because, as observed above, when $\alpha=2$ then the $\alpha$-Cheeger sets are all and only the balls, of any radius, contained in $\Omega$).
\end{rem}

\bibliographystyle{plain}

\bibliography{alpach}

\end{document}

%% file: Rectangles.pdf_t
\begin{picture}(0,0)%
\includegraphics{Rectangles.pdf}%
\end{picture}%
\setlength{\unitlength}{1579sp}%
\begingroup\makeatletter\ifx\SetFigFont\undefined%
\gdef\SetFigFont#1#2#3#4#5{%
  \reset@font\fontsize{#1}{#2pt}%
  \fontfamily{#3}\fontseries{#4}\fontshape{#5}%
  \selectfont}%
\fi\endgroup%
\begin{picture}(12648,3120)(386,-2694)
\put(12251, 39){\makebox(0,0)[lb]{\smash{{\SetFigFont{12}{14.4}{\rmdefault}{\mddefault}{\updefault}$R_L$}}}}
\put(401,-1211){\makebox(0,0)[lb]{\smash{{\SetFigFont{12}{14.4}{\rmdefault}{\mddefault}{\updefault}$P$}}}}
\put(3251, 39){\makebox(0,0)[lb]{\smash{{\SetFigFont{12}{14.4}{\rmdefault}{\mddefault}{\updefault}$R_L$}}}}
\put(7551,-1711){\makebox(0,0)[lb]{\smash{{\SetFigFont{12}{14.4}{\rmdefault}{\mddefault}{\updefault}$P$}}}}
\put(11851,-2361){\makebox(0,0)[lb]{\smash{{\SetFigFont{12}{14.4}{\rmdefault}{\mddefault}{\updefault}$\C^\alpha_L$}}}}
\put(2901,-2361){\makebox(0,0)[lb]{\smash{{\SetFigFont{12}{14.4}{\rmdefault}{\mddefault}{\updefault}$\C^\alpha_L$}}}}
\end{picture}%

%% file: alfaCheeger_revised_pp.bbl
\begin{thebibliography}{10}

\bibitem{AltCas}
F.~Alter and V.~Caselles.
\newblock Uniqueness of the {C}heeger set of a convex body.
\newblock {\em Nonlinear Anal.}, 70(1):32--44, 2009.

\bibitem{AmFuPa}
L.~Ambrosio, N.~Fusco, and D.~Pallara.
\newblock {\em Functions of {B}ounded {V}ariation and {F}ree {D}iscountinuity
  {P}roblems}.
\newblock Oxford University Press, 2000.

\bibitem{DucExn}
P.~Duclos and P.~Exner.
\newblock Curvature-induced bound states in quantum waveguides in two and three
  dimensions.
\newblock {\em Rev. Math. Phys.}, 7:73--102, 1995.

\bibitem{FiMaPr}
A.~Figalli, F.~Maggi, and A.~Pratelli.
\newblock A note on {Cheeger sets}.
\newblock {\em Proceedings of the American Mathematical Society},
  137:2057--2062, 2009.

\bibitem{FMP09}
N.~Fusco, F.~Maggi, and A.~Pratelli.
\newblock Stability estimates for certain {F}aber-{K}rahn, isocapacitary and
  {C}heeger inequalities.
\newblock {\em Ann. Sc. Norm. Super. Pisa Cl. Sci.}, 8(1):51--71, 2009.

\bibitem{GoMaTa2}
E.~Gonzales, U.~Massari, and I.~Tamanini.
\newblock {Minimal Boundaries enclosing a given Volume}.
\newblock {\em Manuscripta Mathematica}, 34:381--395, 1981.

\bibitem{KawLR}
B.~Kawohl and T.~Lachand-Robert.
\newblock Characterization of {C}heeger sets for convex subsets of the plane.
\newblock {\em Pacific J. Math.}, 225:103--118, 2006.

\bibitem{KreKri}
D.~Krej{\v c}i{\v r}{\'\i}k and J.~K{\v r}{\'\i}{\v z}.
\newblock On the spectrum of curved quantum waveguides.
\newblock {\em Publ. RIMS}, 41:757--791, 2005.

\bibitem{KrePra}
D.~Krej{\v c}i{\v r}{\'\i}k and A.~Pratelli.
\newblock The {C}heeger constant of curved strips.
\newblock {\em Pacific J. Math.}, 254:309--333, 2011.

\bibitem{Leonardi}
G.P. Leonardi.
\newblock An overview on the {C}heeger problem.
\newblock {\em in ``New Trends in Shape Optimization'', Int. Series Numerical
  Math. (Birk\"auser)}, (166):117--140, 2015.

\bibitem{LeoPra}
G.P. Leonardi and A.~Pratelli.
\newblock On the {C}heeger sets in strips and non-convex domains.
\newblock {\em Calc. Var. Partial Differential Equations}, 2014.
\newblock to appear.

\bibitem{Parini}
E.~Parini.
\newblock {An Introduction to the Cheeger Problem}.
\newblock {\em Surv. Math. Appl.}, 6:9--22, 2011.

\bibitem{StrZie}
E.~Stredulinsky and W.~P. Ziemer.
\newblock Area minimizing sets subject to a volume constraint in a convex set.
\newblock {\em J. Geom. Anal.}, 7(1):653--677, 1997.

\end{thebibliography}
